\documentclass[11pt,reqno]{article}

\usepackage{mathrsfs, amsfonts}
\usepackage{amsmath}
\usepackage{amssymb}
\usepackage{amsopn}
\usepackage{amsthm}
\usepackage{amstext}

\RequirePackage{geometry}
\newtheorem{thm}{Theorem}[section]
\newtheorem{cor}[thm]{Corollary}
\newtheorem{lem}[thm]{Lemma}
\newtheorem{prop}[thm]{Proposition}
\theoremstyle{remark}
\newtheorem{defn}[thm]{\bf Definition}

\numberwithin{equation}{section}

\RequirePackage{geometry}
\def\dist{\mathrm{dist}}

\def\R{\mathcal R}
\def\N{\mathcal N}
\def\RR{\mathbb R}

\title{Perturbation analysis of bounded homogeneous generalized inverses on Banach spaces}

\author{Jianbing Cao\thanks{Email: caocjb@163.com} \\
Department of mathematics, Henan Institute of Science and Technology\\
Xinxiang, Henan, 453003, P.R. China\\
Department of Mathematics, East China Normal University,\\
Shanghai 200241, P.R. China\\
\and
Yifeng Xue\thanks{Email: yfxue@math.ecnu.edu.cn; Corresponding author}\\
Department of Mathematics, East China Normal University,\\
Shanghai 200241, P.R. China }
\date{}

\begin{document}

\maketitle

\begin{abstract}
Let $X, Y$ be Banach spaces and $T : X \to Y$ be a bounded linear operator. In this paper, we initiate the study of the perturbation
problems for bounded homogeneous generalized inverse $T^h$ and quasi--linear projector generalized inverse $T^H$ of $T$. Some applications
to the representations and perturbations of the Moore--Penrose metric generalized inverse $T^M$ of $T$ are also given.
The obtained results in this paper extend some well--known results for linear operator generalized inverses in this field.
\vspace{3mm}


\end{abstract}

\section{Introduction}
\setcounter{equation}{0}

The expression and perturbation analysis of the generalized inverses (resp. the Moore--Penrose inverses) of bounded
linear operators on Banach spaces (resp. Hilbert spaces) have been widely studied since Nashed's book \cite{NV1} was
published in 1976. Ten years ago, Chen and Xue proposed a notation so--called the stable perturbation of a bounded
operator instead of the rank--preserving perturbation of a matrix in \cite{CX1}. Using this new notation, they established
the perturbation analyses for the Moore--Penrose inverse and the least square problem on Hilbert spaces in
\cite{CWX,CX2,XC}. Meanwhile, Castro--Gonz\'alez and Koliha established the perturbation analysis for Drazin
inverse by using of the gap--function in \cite{GKR, GKW,K}. Later, some of their results were generalized by Chen and Xue
in \cite{X1,X2} in terms of stable perturbation.

Throughout this paper, $X, Y$ are always Banach spaces over real field $\RR$ and $B(X,Y)$ is the Banach space
consisting of bounded linear operators from $X$ to $Y$. For $T\in B(X,Y)$, let $\N(T)$ (resp. $\R(T)$) denote the null space
(resp. range) of $T$. It is well--known that if $\N(T)$ and $\R(T)$ are topologically complemented in the spaces $X$
and $Y$, respectively, then there exists a (projector) generalized inverse $T^+\in B(Y,X)$ of $T$ such that
$$
TT^+T = T, \quad T^+TT^+= T^+,\quad T^+T = I_X-P_{\N(T)}, \quad TT^+= Q_{\R(T)},
$$
where $P_{\N(T)}$ and $Q_{\R(T)}$ are the bounded linear projectors from $X$ and $Y$ onto $\N(T)$ and $\R(T)$,
respectively (cf. \cite{CWX, NV1, XYF1}). But, in general, not every closed subspace in a Banach space is
complemented. Thus the linear generalized inverse $T^+$ of $T$ may not exist. In this case, we may seek other types of
generalized inverses for $T$. Motivated by the ideas of linear generalized inverses and metric generalized inverses
(cf. \cite{NV1, WYW1}), by using the so--called homogeneous (resp. quasi--linear) projector in Banach space, Wang and Li defined the homogeneous (resp. quasi--linear) generalized inverse in \cite{WLS1}. Then,
some further study on these types of generalized inverses in Banach space was given in \cite{BWLX1,WLP1}. More important, from the results in \cite{WLP1,WYW1}, we know that, in some reflexive Banach spaces $X$ and $Y$, for an operator $T \in B(X,Y)$, there may exists a bounded quasi--linear (projector) generalized inverse of $T$, which is generally neither linear nor metric generalized inverse of $T$. So, from this point of view, it is important and necessary to study homogeneous and quasi--linear (projector) generalized inverses in Banach spaces.

Since the homogeneous (or quasi--linear) projector in Banach space are no longer linear, the linear projector generalized
inverse and the homogeneous (or quasi--linear) projector generalized inverse of linear operator in Banach spaces are
quite different. Motivated by the new perturbation results of closed linear generalized inverses \cite{DX2}, in this paper,
we initiate the study of the following problems for bounded homogeneous (resp. quasi--linear projector) generalized inverse: let
$T \in B(X,Y)$ with a bounded homogeneous (resp. quasi--linear projector) generalized inverse $T^h$ (resp. $T^H$), what conditions on the small perturbation
$\delta T$ can guarantee that the bounded homogeneous (resp. quasi--linear projector) generalized inverse $\bar{T}^h$ (resp. $\bar{T}^H$) of the perturbed operator
 $\bar{T} = T+ \delta T$ exists? Furthermore, if it exists, when does $\bar{T}^h$ (resp. $\bar{T}^H$) have the simplest expression
 $(I_X + T^h\delta T)^{-1}T^h$ (resp. $(I_X + T^H\delta T)^{-1}T^H$)? With the concept of the quasi--additivity and the notation of stable perturbation in
 \cite{CX1}, we will present some perturbation results on homogeneous generalized inverses and quasi--linear projector
 generalized inverses in Banach spaces. Explicit representation and perturbation for the Moore--Penrose metric
 generalized inverse of the perturbed operator are also given.

\section{Preliminaries}
\setcounter{equation}{0}

Let $T\in B(X,Y)\backslash\{0\}$. The reduced minimum module $\gamma(T)$ of $T$ is given by
\begin{equation}\label{eqaa}
\gamma(T)=\inf\{\|Tx\|\,\vert\,x\in X, \dist(x,\N(T))=1\},
\end{equation}
where $\dist(x,\N(T))=\inf\{\|x-z\|\,\vert\,z\in\N(T)\}$. It is well--known that $\R(T)$ is closed in $Y$ iff $\gamma(T)
>0$ (cf. \cite{TK,X2}). From (\ref{eqaa}), we can obtain useful inequality as follows:
$$
\|Tx\|\geq\gamma(T)\,\dist(x,\N(T)),\quad\forall\,x\in X.
$$

Recall from \cite{BWLX1,WP11} that a subset $D$ in $X$ is called to be homogeneous if $\lambda\,x\in D$ whenever
$x\in D$ and $\lambda\in\mathbb R$; a mapping $T\colon X \rightarrow Y$ is called to be a bounded homogeneous operator
if $T$ maps every bounded set in $X$ into a bounded set in $Y$ and $T(\lambda\, x)=\lambda\, T(x)$ for every $x\in X$ and every
$\lambda\in\mathbb R$.

Let $H(X,Y)$ denote the set of all bounded homogeneous operators from $X$ to $Y$. Equipped with the usual linear
operations on $H(X,Y)$ and norm on $ T\in H(X,Y)$ defined by $\|T\|=\sup\{\|Tx\|\,|\, \|x\|=1, x\in X\}$, we can easily
prove that $(H(X,Y), \|\cdot\|)$ is a Banach space (cf. \cite{WYW1,WP11}).

\begin{defn}
Let $M$ be a subset of $X$ and $T\colon X \rightarrow Y$ be a mapping. We call $T$ is quasi--additive on $M$ if $T$
satisfies
$$
T(x+z)=T(x) +T(z), \qquad \forall\; x\in X, \;\forall\; z\in M.
$$
\end{defn}

Now we give the concept of quasi--linear projector in Banach spaces.

\begin{defn}[cf. \cite{WLP1,WYW1}]\label{mdef1.1}
Let $P\in H(X,X)$. If $P^2=P$, we call $P$ is a homogeneous projector. In addition, if $P$ is also quasi--additive on $\R(P)$,
i.e., for any $x \in X$ and any $z\in \R(P)$,
  $$
  P(x + z) = P(x) + P(z) = P(x) + z,
  $$
then we call $P$ is a quasi--linear projector.
\end{defn}

Clearly, from Definition \ref{mdef1.1}, we see that the bounded linear projectors, orthogonal projectors in Hilbert spaces  are all quasi--linear projector.

Let $P\in H(X,X)$ be a quasi--linear projector. Then by \cite[Lemma 2.5]{WLP1}, $\R(P)$ is a closed linear subspace
of $X$ and $\R(I-P)=\N(P)$. Thus, we can define ``the quasi--linearly complement'' of a closed linear subspace as follows.
Let $V$ be a closed subspace of $X$. If there exists a bounded quasi--linear projector $P$ on $X$ such
that $V = \R(P)$, then $V$ is said to be bounded quasi--linearly complemented in $X$ and $\N(P)$ is the bounded
quasi--linear complement of $V$ in $X$. In this case, as usual, we may write $X=V\dotplus \N(P)$,
where $\N(P)$ is a homogeneous subset of $X$ and ``$\dotplus$'' means that $V\cap\N(P)=\{0\}$ and $X=V+\N(P)$.

\begin{defn}\label{maindef11.7}
Let $T \in B(X, Y)$. If there is $T^h\in H(Y, X)$ such that
$$
TT^hT=T,\ \quad T^hTT^h=T^h,
$$
then we call $T^h$ is a bounded homogeneous generalized inverse of $T$. Furthermore, if $T^h$ is also quasi--additive on $\mathcal{R}(T)$, i.e., for any $y\in Y$ and any $z\in \mathcal{R}(T)$, we have $$T^h(y+z)=T^h(y)+T^h(z),$$ then we call $T^h$ is a bounded quasi--linear generalized inverse of $T$.
\end{defn}

Obviously, the concept of bounded homogeneous (or quasi-linear) generalized inverse is a generalization of bounded linear generalized
inverse.

Definition \ref{maindef11.7} was first given in paper \cite{BWLX1} for linear transformations and bounded linear
operators. The existence of a homogeneous generalized inverse of $T\in B(X,Y)$ is also given in \cite{BWLX1}. In the following,
we will give a new proof of the existence of a homogeneous generalized inverse of a bounded linear operator.

\begin{prop}\label{prop11.7}
Let $T \in B(X, Y)\backslash\{0\}$. Then $T$ has a homogeneous generalized inverse $T^h\in H(Y,X)$ iff  $\R(T)$ is closed and there exist a bounded
quasi--linear projector $P_{\N(T)}\colon X \to \N(T)$ and a bounded homogeneous projector $Q_{\R(T)}: Y \to \R(T)$.
\end{prop}
\begin{proof}
Suppose that there is $T^h\in H(Y,X)$ such that $TT^hT=T$ and $T^hTT^h=T^h$. Put $P_{\N(T)}=I_X-T^hT$ and $Q_{\R(T)}=TT^h$. Then
$P_{\N(T)}\in H(X,X)$, $Q_{\R(T)}\in H(Y,Y)$ and
\begin{align*}
P_{\N(T)}^2&=(I_X-T^hT)(I_X-T^hT)=I_X-T^hT-T^hT(I_X-T^hT)=P_{\N(T)},\\
Q_{\R(T)}^2&=TT^hTT^h=TT^h=Q_{\R(T)}.
\end{align*}
From $TT^hT=T$ and $T^hTT^h=T^h$, we can get that $\N(T)=\R(P_{\N(T)})$ and $\R(T)=\R(Q_{\R(T)})$. Since for any $x\in X$ and any
$z\in\N(T)$,
\begin{align*}
P_{\N(T)}(x+z)&=x+z-T^hT(x+z)=x+z-T^hTx\\&=P_{\N(T)}x+z=P_{\N(T)}x+P_{\N(T)}z,
\end{align*}
it follows that $P_{\N(T)}$ is quasi--linear. Obviously, we see that $Q_{\R(T)}: Y \to \R(T)$ is a bounded homogeneous projector.

Now for any $x\in X$,
$$
\dist(x,\N(T))\leq\|x-P_{\N(T)}x\|=\|T^hTx\|\leq\|T^h\|\|Tx\|.
$$
Thus, $\gamma(T)\geq\dfrac{1}{\|T^h\|}>0$ and hence $\R(T)$ is closed in $Y$.

Conversely, for $x\in X$, let $[x]$ stand for equivalence class of $x$ in $X/\N(T)$. Define mappings $\phi\colon
\R(I-P_{\N(T)})\rightarrow X/\N(T)$ and $\hat T\colon X/\N(T)\rightarrow\R(T)$ respectively, by
$$
\phi(x)=[x],\quad\forall\, x\in\R(I-P_{\N(T)})\ \text{and}\ \hat T([z])=Tz,\quad\forall\,z\in X.
$$
Clearly, $\hat T$ is bijective. Noting that the quotient space $X/\N(T)$ with the norm $\|[x]\|=\dist(x,\N(T))$,
$\forall\,x\in X$, is a Banach space (cf. \cite{XYF1}) and $\|Tx\|\geq\gamma(T)\,\dist(x,\N(T))$ with $\gamma(T)>0$,
$\forall\,x\in X$, we have $\|\hat T[x]\|\geq\gamma(T)\|[x]\|$, $\forall\,x\in X$. Therefore, $\|\hat T^{-1}y\|\leq
\dfrac{1}{\gamma(T)}\|y\|$, $\forall\,y\in\R(T)$.

Since $P_{\N(T)}$ is a quasi--linear projector, it follows that $\phi$ is bijective and $\phi^{-1}([x])=(I-P_{\N(T)})x$,
$\forall\,x\in X$. Obviously, $\phi^{-1}$ is homogeneous and for any $z\in\N(T)$,
$$
\|\phi^{-1}([x])\|=\|(I-P_{\N(T)})(x-z)\|\leq(1+\|P_{\N(T)}\|)\|x-z\|
$$
which implies that $\|\phi^{-1}\|\leq 1+\|P_{\N(T)}\|$. Put $T_0=\hat T\circ\phi\colon\R(I-P_{\N(T)})\rightarrow\R(T)$.
Then $T_0^{-1}=\phi^{-1}\circ\hat T^{-1}\colon\R(T)\rightarrow\R(I-P_{\N(T)})$ is homogeneous and bounded with
$\|T_0^{-1}\|\leq \gamma(T)^{-1}(1+\|P_{\N(T)}\|)$. Set $T^h=(I-P_{\N(T)})T_0^{-1}Q_{\R(T)}$. Then $T^h\in H(Y,X)$ and
$$
TT^hT=T,\ T^hTT^h=T^h,\ TT^h=Q_{\R(T)},\ T^hT=I_X-P_{\N(T)}.
$$
This finishes the proof.
\end{proof}

Recall that a closed subspace $V$ in $X$ is Chebyshev if for any $x\in X$, there is a unique $x_0\in V$ such that
$\|x-x_0\|=\dist(x,V)$. Thus, for the closed Chebyshev space $V$, we can define a mapping $\pi_V\colon X\rightarrow V$ by
$\pi_V(x)=x_0$. $\pi_V$ is called to be the metric projector from $X$ onto $V$. From \cite{WYW1}, we know that $\pi_V$ is a quasi--linear projector with $\|\pi_V\|\le 2$.
Then by Proposition
\ref{prop11.7}, we have

\begin{cor}[\cite{NRX1,WYW1}]\label{cor2a}
Let $T\in B(X,Y)\backslash\{0\}$ with $\R(T)$ closed. Assume that $\N(T)$ and $\R(T)$ are Chebyshev
subspaces in $X$ and $Y$, respectively. Then there is $T^h\in H(Y,X)$ such that
\begin{equation}\label{eqbb}
TT^hT=T,\ T^hTT^h=T^h,\ TT^h=\pi_{\R(T)},\ T^hT=I_X-\pi_{\N(T)}.
\end{equation}
\end{cor}

The bounded homogeneous generalized inverse $T^h$ in \eqref{eqbb} is called to be the Moore--Penrose metric generalized
inverse of $T$. Such $T^h$ in (\ref{eqbb}) is unique and is denoted by $T^M$ (cf. \cite{WYW1}).

\begin{cor}\label{cor2b}
Let $T\in B(X,Y)\backslash\{0\}$ such that the bounded homogeneous generalized inverse $T^h$ exists. Assume that $\N(T)$ and $\R(T)$ are Chebyshev
subspaces in $X$ and $Y$, respectively. Then $T^M=(I_X-\pi_{\N(T)})T^h\pi_{\R(T)}$.
\end{cor}

\begin{proof}
Since $\N(T)$ and $\R(T)$ are Chebyshev subspaces, it follows from Corollary \ref{cor2a} that $T$ has the unique
Moore--Penrose metric generalized inverse $T^M$ which satisfy
$$
TT^MT=T,\ T^MTT^M=T^M,\ TT^M=\pi_{\R(T)},\ T^MT=I_X-\pi_{\N(T)}.
$$
Set $T^\natural=(I_X-\pi_{\N(T)})T^h\pi_{\R(T)}$. Then $T^\natural=T^MTT^hTT^M=T^MTT^M=T^M.$
\end{proof}

\section{\bf Perturbations for bounded homogeneous generalized inverse}
\setcounter{equation}{0}

In this section, we extend some perturbation results of linear generalized inverses to bounded homogeneous generalized inverses.
We start our investigation with some lemmas, which are prepared for the proof of our  main results. The following
result is well--known for bounded linear operators, we generalize it to the bounded homogeneous operators in the following
form.

\begin{lem}\label{qlem2.8}
Let $T\in H(X,Y)$ and $S\in H(Y,X)$ such that $T$ is quasi--additive on $\R(S)$ and $S$ is quasi--additive on $\R(T)$,
then $I_Y+TS$ is invertible in $H(Y,Y)$ if and only if $I_X+ST$ is invertible in $H(X,X)$.
\end{lem}

\begin{proof}
If there is a $\Phi \in H(Y,Y)$ be such that $(I_Y+TS)\Phi =\Phi(I_Y+TS)=I_Y$, then
\begin{align*}
I_X&=I_X+ST-ST=I_X+ST-S((I_Y+TS)\Phi)T\\
&=I_X+ST-((S+STS)\Phi)T \quad(S\ \text{quasi-additive on}\ \R(T))\\
&=I_X+ST-((I_X+ST)S\Phi)T\\
&=(I_X+ST)(1_X-S\Phi T) \quad(T\ \text{quasi--additive on}\ \R(S)).
\end{align*}
Similarly, we also have $I_X=(I_X-S\Phi T)(I_X+ST)$. Thus, $I_X+ST$ is invertible on $X$ with $(I_X+ST)^{-1}
=(1_X-S\Phi T)\in H(X,X)$.

The converse can also be proved by using the same way as above.
\end{proof}

\begin{lem}\label{qlem11.3}
Let $T\in B(X,Y)$ such that $T^h\in H(Y,X)$ exists and let $\delta T\in B(X,Y)$ such that $T^h$ is quasi--additive on
$\R(\delta T)$ and $(I_X+T^h\delta T)$ is invertible in $B(X,X)$. Then $I_Y+\delta TT^h: Y\rightarrow Y$ is invertible
in $H(Y,Y)$ and
\begin{equation}\label{eqcc}
\Phi=T^h(I_Y+\delta TT^h)^{-1}=(I_X+T^h\delta T)^{-1}T^h
\end{equation}
is a bounded homogeneous operator with $\R(\Phi)=\R(T^h)$ and $\N(\Phi)=\N(T^h)$.
\end{lem}

\begin{proof}
By Lemma \ref{qlem2.8}, $I_Y+\delta TT^h: Y\rightarrow Y$ is invertible in $H(Y,Y)$.

Clearly, $I_X+T^h\delta T$ is linear bounded operator and $I_Y+\delta TT^h\in H(Y,Y)$. From the equation
$$
(I_X+T^h\delta T)T^h=T^h(I_Y+\delta TT^h)
$$
and $T^h \in H(Y, X)$, we get that $\Phi$ is a bounded homogeneous operator. Finally, from (\ref{eqcc}),
we can obtain that $\R(\Phi)=\R(T^h)$ and $\N(\Phi)=\N(T^h)$.
\end{proof}

Recall from \cite{CX1} that for $T\in B(X,Y)$ with bounded linear generalized inverse $T^+\in B(Y,X)$,
we say that $\bar T=T+\delta T\in B(X,Y)$ is a stable perturbation of $T$ if $\R(\bar T)\cap\N(T^+)=\{0\}$.
Now for $T\in B(X,Y)$ with $T^h\in H(Y,X)$, we also say that $\bar T=T+\delta T\in B(X,Y)$ is a stable perturbation of $T$ if
$\R(\bar T)\cap\N(T^h)=\{0\}$.

\begin{lem}\label{lem1.14}
Let $T\in B(X,Y)$ such that $T^h\in H(Y,X)$ exists. Suppose that $\delta T\in B(X,Y)$ such that $T^h$ is quasi--additive on
$\R(\delta T)$ and $I_X+T^h\delta T$ is invertible in $B(X,X)$ Put $\bar{T}=T+\delta T$. If
$\R(\bar{T})\cap \N(T^h)=\{0\}$, then
$$
\N(\bar T)=(I_X+T^h\delta T)^{-1}\N(T)\ \text{and}\ \R(\bar{T})=(I_Y+\delta TT^h)\R(T).
$$
\end{lem}

\begin{proof}
Set $P=(I_X+T^h\delta T)^{-1}(I_X-T^hT)$. We first show that $P^2 =P$ and $\R(P)=\N(\bar{T})$. Since $T^hTT^h=T^h$, we get $(I_X-T^hT)T^h \delta T=0$ and then
\begin{eqnarray}\label{lemeq3.2j}
(I_X-T^hT)(I_X+T^h\delta T)=I_X-T^hT
\end{eqnarray}
and so that
\begin{eqnarray}\label{lemeq3.3j}
I_X-T^hT=(I_X-T^hT)(I_X+T^h\delta T)^{-1}.
\end{eqnarray}
Now, by using \eqref{lemeq3.2j} and \eqref{lemeq3.3j}, it is easy to get $P^2=P$.

Since $T^h$ is quasi--additive on $\R(\delta T)$, we see $I_X- T^hT = (I_X+ T^h\delta T)-T^h \bar{T}$.
Then for any $x\in X$, we have
\begin{align}\label{lemeq3.3}
Px&=(I_X+T^h\delta T)^{-1}(I_X- T^hT)x \nonumber\\
&=(I_X+T^h\delta T)^{-1}[(I_X+ T^h\delta T)-T^h \bar{T}]x\nonumber\\
&=x-(I_X+T^h\delta T)^{-1}T^h\bar{T}x.
\end{align}
From \eqref{lemeq3.3}, we get that if $x\in \N(\bar{T})$, then $x\in \R(P)$. Thus, $\N(\bar{T})\subset \R(P)$.

 Conversely, let $z\in \R(P)$, then $z=Pz$. From \eqref{lemeq3.3}, we get $(I_X+T^h\delta T)^{-1}T^h\bar{T}x=0$.
 Therefore, we have $\bar{T}x\in \R(\bar{T})\cap\N(T^h)=\{0\}$. Thus, $x\in \N(\bar{T})$ and then $\R(P)=\N(\bar{T})$.

From the Definition of $T^h$, we have $\N(T)=\R(I_X- T^hT)$. Thus,
$$
(I_X+T^h\delta T)^{-1}\N(T)=(I_X+T^h\delta T)^{-1}\R(I_X- T^hT)=\R(P)=\N(\bar{T}).
$$

Now, we prove that $\R(\bar{T})=(I_Y+\delta TT^h)\R(T)$. From $(I_Y + \delta TT^h)T = \bar{T}T^hT$, we get that
$(I_Y + \delta TT^h)\R(T)\subset \R(\bar{T})$. On the other hand, since $T^h$ is quasi--additive on $\R(\delta T)$ and
$\R(P)=\N(\bar{T})$, we have for any $x\in X$,
\begin{align}\label{lemeq3.5j}
0&=\bar{T}Px=\bar{T}(I_X+T^h\delta T)^{-1}(I_X-T^hT)x\nonumber\\
&=\bar{T}x -\bar{T}(I_X+T^h\delta T)^{-1}(T^h\delta Tx+ T^h{T}x)\nonumber\\
&=\bar{T}x -\bar{T}(I_X+ T^h\delta T)^{-1}T^h\bar{T}x=\bar{T}x -\bar{T}T^h(I_Y+ \delta T T^h)^{-1}\bar{T}x\nonumber\\
&=\bar Tx-(I_Y+\delta TT^h-I_Y+TT^h)(I_Y+ \delta T T^h)^{-1}\bar{T}x\nonumber\\
&=(I_Y-T T^h)(I_Y+ \delta T T^h)^{-1}\bar{T}x.
\end{align}
Since $\N(I_Y-T T^h)=\R(T)$, it follows \eqref{lemeq3.5j} that $(I_Y+ \delta T T^h)^{-1}\R(\bar{T})\subset \R(T)$,
that is, $\R(\bar{T})\subset (I_Y+ \delta T T^h)\R(T)$. Consequently, $\R(\bar{T})=(I_Y+\delta TT^h)\R(T)$.
\end{proof}

Now we can present the main perturbation result for bounded homogeneous generalized inverse on Banach spaces.

\begin{thm}\label{mthm1.15}
Let $T\in B(X,Y)$ such that $T^h\in H(Y,X)$ exists. Suppose that $\delta T\in B(X,Y)$ such that $T^h$ is quasi--additive on
$\R(\delta T)$ and $I_X+T^h\delta T$ is invertible in $B(X,X)$. Put $\bar{T}=T+\delta T$. Then the following statements
are equivalent:
\begin{enumerate}
  \item[$(1)$] $\Phi=T^h(I_Y+\delta T T^h)^{-1}$ is a bounded homogeneous generalized inverse of $\bar{T}$;
  \item[$(2)$] $\R(\bar{T})\cap \N(T^h)=\{0\}$;
  \item[$(3)$] $\R(\bar{T})=(I_Y+\delta TT^h)\R(T)$;
  \item[$(4)$] $(I_Y+T^h\delta T)\N(\bar{T})=\N(T)$;
  \item[$(5)$] $(I_Y+\delta TT^h)^{-1}\bar{T} \N(T)\subset\R(T)$;
\end{enumerate}
\end{thm}

\begin{proof}
We prove our theorem by showing that
$$
(3)\Rightarrow (5)\Rightarrow (4)\Rightarrow (2)\Rightarrow (3)\Rightarrow
(1)\Rightarrow (3).
$$

$(3)\Rightarrow (5)$ This is obvious since $(I+\delta TT^h)$ is invertible and $\N(T) \subset X$.

$(5)\Rightarrow (4)$. Let $x\in \N(\bar{T})$, then we see
 $(I_X+T^h\delta T)x=x-T^hTx\in \N(T)$. Hence $(I_X+T^h\delta T)\N(\bar{T})\subset
\N(T)$. Now for any $x\in \N(T)$, then by (5), there exists some $z\in X$ such that $\bar{T}x =(I_Y+\delta TT^h)Tz=
\bar{T}T^hTz.$ So $x-T^hTz\in \N(\bar{T})$ and hence
$$
(I_X+T^h\delta T)(x-T^hTz)=(I_X-T^hT)(x-T^hTz)=x.
$$
Consequently, $(I_X+T^h\delta T)\N(\bar{T})=\N(T)$.

 $(4)\Rightarrow (2)$.
Let $y\in R(\overline{T})\cap N(T^h)$, then there exists an $x\in X$ such that $y=\bar{T}x$ and $T^h\bar{T}x=0$. We can check that
$$T(I_X+T^h\delta T)x=Tx+TT^h\delta Tx=Tx+TT^h\bar{T}x-TT^hTx=0.
$$
Thus, $(I_X+T^h\delta T)x\in \N(T)$. By (4), $x\in \N(\bar{T})$ and so that $y=\bar{T}x=0$.

$(2)\Rightarrow (3)$ follows from Lemma \ref{lem1.14}.

$(3)\Rightarrow (1)$ Noting that by Lemma \ref{qlem11.3}, we have $$\Phi=T^h(I_Y+\delta
TT^h)^{-1}=(I_X+T^h\delta
T)^{-1}T^h$$
is a bounded homogeneous operator with $\R(\Phi)=\R(T^h)$ and $\N(\Phi)=\N(T^h)$.
We need to prove that $\Phi \bar{T} \Phi=\Phi$ and $\bar{T}\Phi\bar{T}=\bar{T}$. Since $T^h$ is quasi--additive
on $\R(\delta T)$, we have $T^h\bar T=T^hT+T^h\delta T$. Therefore,
\begin{align*}
\Phi \bar{T} \Phi&=(I_Y+T^h\delta T)^{-1}T^h\bar{T}(I_Y+T^h\delta T)^{-1}T^h\\
&=(I_Y+T^h\delta T)^{-1}T^h[(I_X+T^h\delta T)-(I_X-T^hT)](I_X+T^h\delta T)^{-1}T^h\\
&=(I_X+T^h\delta T)^{-1}T^h-(I_X+T^h\delta T)^{-1}(I_X-T^hT)(I_X+T^h\delta T)^{-1}T^h\\
&=(I_X+T^h\delta T)^{-1}T^h-(I_X+T^h\delta T)^{-1}(I_X-T^hT)T^h(I_X+\delta TT^h)^{-1}\\
&=\Phi.
\end{align*}
$\R(\bar{T})=(I_Y+\delta TT^h)\R(T)$ means that $(I_Y-TT^h)(I_Y+\delta TT^h)^{-1}\bar T=0$. So
\begin{align*}
\bar T\Phi\bar T&=(T+\delta T)T^h(I_Y+T^h\delta T)^{-1}\bar T\\
&=(I_Y+\delta TT^h+TT^h-I_Y)(I_Y+T^h\delta T)^{-1}\bar T\\
&=\bar T.
\end{align*}

$(1)\Rightarrow (3)$ From $\bar T\Phi\bar T=\bar T$, we have $(I_Y-TT^h)(I_Y+\delta TT^h)^{-1}\bar T=0$ by the proof
of $(3)\Rightarrow (1)$. Thus,
$(I_Y+\delta TT^h)^{-1}\R(\bar T)\subset\R(T)$.  From $(I_Y + \delta TT^h)T = \bar{T}T^hT$, we get that
$(I_Y + \delta TT^h)\R(T)\subset \R(\bar{T})$. So $(I_Y + \delta TT^h)\R(T)=\R(\bar{T})$.
\end{proof}

\begin{cor}\label{corth3.6}
Let $T\in B(X,Y)$ such that $T^h\in H(Y,X)$ exists. Suppose that $\delta T\in B(X,Y)$ such that $T^h$ is quasi--additive on
$\R(\delta T)$ and $\|T^h\delta T\|<1$. Put $\bar{T}=T+\delta T$. If $\N(T)\subset \N(\delta T)$ or
$\R(\delta T)\subset \R(T)$, then $\bar{T}$ has a homogeneous bounded generalized inverse
$$
\bar{T}^h=T^h(I_Y+\delta TT^h)^{-1}=(I_X+T^h\delta T)^{-1}T^h.
$$
\end{cor}

\begin{proof}
If $N(T)\subset N(\delta T)$, then $N(T)\subset
N(\bar{T})$. So Condition (5) of Theorem \ref{mthm1.15} holds. If $\R(\delta T)\subset R(T)$, then
$R(\bar{T})\subset \R(T)$. So $\R(\bar T)\cap\N(T)\subset\R(T)\cap\N(T^h)=\{0\}$ and consequently, $\bar T$ has the
homogeneous bounded generalized inverse $T^h(I_Y+\delta TT^h)^{-1}=(I_X+T^h\delta T)^{-1}T^h$ by Theorem \ref{mthm1.15}.
\end{proof}

\begin{prop}
Let $T\in B(X,Y)$ with $\R(T)$ closed. Assume that $\N(T)$ and $\R(T)$ are Chebyshev subspaces in $X$ and $Y$, respectively.
Let $\delta T\in B(X,Y)$ such that $T^M$ is quasi--additive on $\R(\delta T)$ and $\|T^M\delta T\|<1$. Put
$\bar{T}=T+\delta T$. Suppose that $\N(\bar T)$ and $\overline{\R(\bar T)}$ are Chebyshev subspaces in $X$ and $Y$,
respectively. If $\R(\bar T)\cap\N(T^M)=\{0\}$, then $\R(\bar T)$ is closed in $Y$ and $\bar T$ has the
Moore--Penrose metric generalized inverse
$$
\bar T^M=(I_X-\pi_{\N(\bar T)})(I_X+T^M\delta T)^{-1}T^M\pi_{\R(\bar T)}
$$
with $\|\bar T^M\|\leq\dfrac{2\|T^M\|}{1-\|T^M\delta T\|}$.
\end{prop}
\begin{proof}$T^M$ exists by Corollary \ref{cor2a}. Since $T^M\delta T$ is $\mathbb R$--linear and $\|T^M\delta T\|<1$,
we have $I_X+T^M\delta T$ is invertible in $B(X,X)$. By Theorem \ref{mthm1.15} and Proposition \ref{prop11.7},
$\R(\bar T)\cap\N(T^M)=\{0\}$ implies that $\R(\bar T)$ is closed and $\bar T$ has a bounded homogeneous generalized
inverse $\bar T^h=(I_X+T^M\delta T)^{-1}T^M$. Then by Corollary \ref{cor2b}, $\bar T^M$ has the form
$$
\bar T^M=(I_X-\pi_{\N(\bar T)})(I_X+T^M\delta T)^{-1}T^M\pi_{\R(\bar T)}.
$$
Note that $\|x-\pi_{\N(\bar T)}x\|=\dist(x,\N(\bar T))\leq\|x\|$, $\forall\,x\in X$. So $\|I_X-\pi_{\N(\bar T)}\|\leq 1$.
Therefore,
$$
\|\bar T^M\|\leq\|I_X-\pi_{\N(\bar T)}\|\|(I_X+T^M\delta T)^{-1}T^M\|\|\pi_{\R(\bar T)}\|
\leq\frac{2\|T^M\|}{1-\|T^M\delta T\|}.
$$
This completes the proof.
\end{proof}

\section{Perturbation for quasi--linear projector generalized inverse}
\setcounter{equation}{0}

We have known that the range of a bounded qausi--linear projector on a Banach space is closed(see \cite[Lemma 2.5]{WLP1}). Thus, from Definition \ref{maindef11.7} and the proof of Proposition \ref{prop11.7}, the following result is obvious.

\begin{prop}\label{corfj11.7}
Let $T \in B(X, Y)\backslash\{0\}$. Then $T$ has a bounded quasi--linear generalized inverse $T^h\in H(Y,X)$ iff there exist a bounded linear projector $P_{\N(T)}\colon X \to \N(T)$ and a bounded quasi--linear projector $Q_{\R(T)}: Y \to \R(T)$.
\end{prop}

Motivated by related results in papers \cite{BWLX1,WLP1,WLS1} and the definition of the oblique projection generalized inverses in Banach space(see \cite{NV1,XYF1}), based on Proposition \ref{corfj11.7}, we can give the following definition of quasi--linear projector generalized inverse of a bounded linear operator on Banach space.

\begin{defn}\label{quasidef1.1}
Let $T \in B(X,Y)$. Let $T^H \in H(Y,X)$ be a bounded homogeneous operator. If there exist a bounded linear projector $P_{\N(T)}$ from $X$ onto $\N(T)$ and a bounded quasi--linear projector $Q_{\R(T)}$ from $Y$ onto $\R(T)$, respectively, such that
\begin{eqnarray*}
(1)\, TT^HT = T; \quad (2)\, T^HTT^H= T^H; \quad  (3)\, T^HT = I_X-P_{\N(T)}; \quad  (4)\, TT^H= Q_{\R(T)}.
\end{eqnarray*}
Then $T^H$ is called a quasi--linear projector generalized inverse of $T$.
\end{defn}

For $T\in B(X,Y)$, if $T^H$ exists, then from Proposition \ref{corfj11.7} and Definition \ref{maindef11.7}, we see that
$\R(T)$ is closed and $T^H$ is quasi--additive on $\R(T)$, in this case, we may call $T^H$ is a quasi--linear operator. Choose $\delta T\in B(X,Y)$ such that $T^H$ is also
quasi--additive on $\R(\delta T)$, then $I_X+T^H\delta T$ is a bounded linear operator and
$I_Y+\delta TT^H$ is a bounded linear operator on $\R(\bar T)$.

\begin{lem}\label{mplem2.1}
Let $T\in B(X,Y)$ such that $T^H$ exists and let $\delta T\in B(X,Y)$ such that $T^H$ is quasi--additive on $\R(\delta T)$.
Put $\bar{T}=T+\delta T$. Assumes that $X=\N(\bar{T})\dotplus \R(T^H)$ and $Y=\R(\bar{T})\dotplus\N(T^H)$. Then
\begin{enumerate}
  \item[$(1)$] $I_X+T^H\delta T: X\rightarrow X$ is a invertible bounded linear operator;
  \item[$(2)$] $I_Y+\delta T T^H: Y\rightarrow Y$ is a invertible quasi--linear operator;
   \item[$(3)$] $\Upsilon=T^H(I_Y+\delta TT^H)^{-1}=(I_X+T^H\delta T)^{-1}T^H$ is a bounded homogeneous operator.
\end{enumerate}
\end{lem}

\begin{proof}
Since $I_X+T^H\delta T\in B(X,X)$, we only need to show that $\N(I_X+T^H\delta T)=\{0\}$ and $\R(I_X+T^H\delta T)=X$
under the assumptions.

We first show that $\N(I_X+T^H\delta T)=\{0\}$. Let $x\in \N(I_X+ T^H\delta T)$, then
$$
(I_X+ T^H\delta T)x=(I_X-T^HT)x+T^H\bar Tx=0
$$
since $T^H$ is quasi--linear. Thus $(I_X-T^HT)x=0=T^H\bar Tx$ and hence $\bar Tx\in\R(\bar T)\cap\N(T^H)$. Noting that
$Y=\R(\bar{T})\dotplus\N(T^H)$, we have $\bar Tx=0$ and hence $x\in\R(T^H)\cap\N(\bar T)$. From
$X=\N(\bar{T})\dotplus \R(T^H)$, we get that $x=0$.

Now, we prove that $\R(I_X+T^H\delta T)=X$. Let $x\in X$ and put $x_1=(I_X-T^HT)x$, $x_2=T^HTx$.
Since $Y=\R(\bar{T})\dotplus \N(T^H)$, we have $\R(T^H)=T^H\R(\bar T)$. Therefore, from $X=\N(\bar{T})\dotplus\R(T^H)$,
we get that $\R(T^H)=T^H\R(\bar T)=T^H\bar T\R(T^H)$. Consequently, there is $z\in Y$ such that $T^H(Tx_2-\bar Tx_1)=
T^H\bar TT^Hz$. Set $y=x_1+T^Hz\in X$.  Noting that $T^H$ is quasi--additive on $\R(T)$ and $\R(\delta T)$, respectively.
we have
\begin{align*}
(I_X+T^H\delta T)y&=(I_X-T^HT+T^H\bar{T})(x_1+T^Hz)\\
&=x_1+T^H\bar Tx_1+T^H\bar TT^Hz\\
&=x_1+T^H\bar Tx_1+T^H(Tx_2-\bar Tx_1)\\
&=x.
\end{align*}
Therefore, $X=\R(I_Y+T^H\delta T)$.

Similar to Lemma \ref{qlem11.3}, we have $\Upsilon=T^H(I_Y+\delta TT^H)^{-1}=(I_X+T^H\delta T)^{-1}T^H$ is a bounded
homogeneous operator.
\end{proof}

\begin{thm}\label{fmthm1.21}
Let $T\in B(X,Y)$ such that $T^H$ exists and let $\delta T\in B(X,Y)$ such that $T^H$ is quasi--additive on
$\R(\delta T)$. Put $\bar{T}=T+\delta T$. Then the following statements are equivalent:
\begin{enumerate}
\item[$(1)$] $I_X+T^H\delta T$ is invertible in $B(X,X)$ and $\R(\bar T)\cap\N(T^H)=\{0\};$
  \item[$(2)$] $I_X+T^H\delta T$ is invertible in $B(X,X)$ and $\Upsilon=T^H(I_Y+\delta TT^H)^{-1}=
  (I_X+T^H\delta T)^{-1}T^H$ is a quasi--linear projector generalized inverse of $\bar{T};$
  \item[$(3)$] $X=\N(\bar{T})\dotplus\R(T^H)$ and $Y=\R(\bar{T})\dotplus\N(T^H)$, i.e., $\N(\bar{T})$ is topological
  complemented in $X$ and $\R(\bar{T})$ is quasi--linearly complemented in $Y$.
\end{enumerate}
\end{thm}

\begin{proof}$(1)\Rightarrow(2)$ By Theorem \ref{mthm1.15}, $\Upsilon=T^H(I_Y+\delta TT^H)^{-1}=(I_X+T^H\delta T)^{-1}T^H$
is a bounded homogeneous generalized inverse of $T$.  Let $y\in Y$ and $z\in\R(\bar T)$. Then $z=Tx+\delta Tx$ for some
$x\in X$. Since $T^H$ is quasi--additive on $\R(T)$ and $\R(\delta T)$, it follows that
$$
T^H(y+z)=T^H(y+Tx+\delta Tx)=T^H(y)+T^H(Tx)+T^H(\delta Tx)=T^Hy+T^Hz,
$$
i.e., $T^H$ is quasi--additive on $\R(\bar T)$ and hence $\Upsilon$ is quasi--linear. Set
$$
\bar{P}=(I_X+T^H\delta T)^{-1}(I_X-T^HT),\qquad \bar{Q}=\bar{T}(I_X+T^H\delta T)^{-1}T^H.
$$
Then, by the proof of Lemma \ref{lem1.14}, $\bar P\in H(X,X)$ is a projector with $\R(\bar P)=\N(\bar T)$. Noting that
$(I_X+T^H\delta T)^{-1}$ and $I_X-T^HT$ are all linear. So $\bar P$ is linear. Furthermore,
\begin{align*}
\Upsilon\bar T\! &=(I_X+T^H\delta T)^{-1}T^H(T+\delta T)\!\\&=(I_X+T^H\delta T)^{-1}(I_X+T^H\delta T+T^HT-I_X)\!\\&=I_X-\bar P.
\end{align*}

Since $T^H$ is quasi--additive on $\R(\bar T)$, it follow that $\bar{Q}=\bar{T}(I+T^H\delta T)^{-1}T^H=\bar T\Upsilon$
is quasi--linear and bounded with $\R(\bar Q)\subset\R(\bar T)$. Noting that
\begin{align*}
\bar Q&=\bar TT^H(I_Y+\delta TT^H)^{-1}=(I_Y+\delta TT^H+TT^H-I_Y)(I_Y+\delta TT^H)^{-1}\\
&=I_Y-(I_Y-TT^H)(I_Y+\delta TT^H)^{-1}
\end{align*}
and $(I_Y+\delta TT^H)^{-1}\R(\bar T)=\R(T)$ by Lemma \ref{lem1.14}, we have $\R(\bar T)=\bar Q(\R(\bar T))\subset
\R(\bar Q)$. Thus, $\R(\bar Q)=\R(\bar T)$. From $\Upsilon\bar T=I_X-\bar{P}$ and $\R(\bar P)=\N(\bar T)$, we see that $\Upsilon \bar{T} \Upsilon=\Upsilon$, then we have $$\bar{Q}^2=\bar{T}(I_X+T^H\delta T)^{-1}T^H\bar{T}(I_X+T^H\delta T)^{-1}T^H=\bar{T} \Upsilon \bar{T} \Upsilon=\bar{Q}.$$ Therefore, by Definition \ref{quasidef1.1}, we get $\bar T^H=\Upsilon$.

$(2)\Rightarrow (3)$ From $\bar T^H=T^H(I_Y+\delta TT^h)^{-1}=(I_X+T^H\delta T)^{-1}T^H$, we obtain that
$\R(\bar T^H)=\R(T^H)$ and $\N(\bar T^H)=\N(T^H)$. From $\bar T\bar T^H\bar T=\bar T$, $\bar T^H\bar T\bar T^H=\bar T^H$,
we get that
$$
\R(I_X-\bar T^H\bar T)=\N(\bar T),\ \R(\bar T^H\bar T)=\R(\bar T^H),\ \R(\bar T\bar T^H)=\R(\bar T),\
\R(I_Y-\bar T\bar T^H)=\N(\bar T^H)
$$
Thus $\R(\bar T^H\bar)=\R(T^H)$ and $\R(I_Y-\bar T\bar T^H)=\N(T^H)$. Therefore,
\begin{align*}
X&=\R(I_X-\bar T^H\bar T)\dotplus\R(\bar T^H\bar T)=\N(\bar T)\dotplus\R(T^H),\\
Y&=\R(\bar T\bar T^H)\dotplus\R(I_Y-\bar T\bar T^H)=\R(\bar T)\dotplus\N(T^H).
\end{align*}

$(3)\Rightarrow(1)$ By Lemma \ref{mplem2.1}, $I_X+T^H\delta T$ is invertible in $H(X,X)$. Now from
$Y=\R(\bar{T})\dotplus\N(T^H)$, we get that $\R(\bar{T})\cap\N(T^H)=\{0\}$.
\end{proof}
\begin{lem}[\cite{CC1}]\label{qlem2.9}
 Let $A\in B(X,X)$. Suppose that there exist two constants $\lambda_1, \lambda_2\in [0,1)$ such that
 $$
 \|Ax\|\leq \lambda_1\|x\|+\lambda_2\|(I+A)x\|,\quad\quad ( \forall \;x\in X).
 $$
 Then $I+A\colon X\rightarrow X$ is bijective. Moreover, for any $x\in X$,
 \begin{eqnarray*}
 \frac{1-\lambda_1}{1+\lambda_2}\|x\|\leq\|(I+A)x\|\leq\frac{1+\lambda_1}{1-\lambda_2}\|x\|,\quad
 \frac{1-\lambda_2}{1+\lambda_1}\|x\|\leq\|(I+A)^{-1}x\|\leq\frac{1+\lambda_2}{1-\lambda_1}\|x\|.
 \end{eqnarray*}
\end{lem}

Let $T\in B(X,Y)$ such that $T^H$ exists. Let $\delta T\in B(X,Y)$ such that $T^H$ is quasi--additive on $\R(\delta T)$
and satisfies
\begin{eqnarray}\label{eq11.2}
\|T^H\delta Tx\|\leq \lambda_1\|x\|+\lambda_2\|(I+T^H\delta T)x\|\quad  (\forall\; x\in X),
\end{eqnarray}
where $\lambda_1, \lambda_2\in [0,1)$.

\begin{cor}\label{mplemfj2.1}
Let $T\in B(X,Y)$ such that $T^H$ exists. Suppose that $\delta T\in B(X,Y)$ such that $T^H$ is quasi--additive on
$\R(\delta T)$ and satisfies \eqref{eq11.2}. Put $\bar{T}=T+\delta T$. Then $I_X+T^H\delta T$ is invertible in $H(X,X)$
and $\bar{T}^H=(I_X+ T^H\delta T)^{-1}T^H$ is well defined with
$$
\dfrac{\|\bar{T}^H-T^H\|}{\|T^H\|}\leq \dfrac{(2+\lambda_1)(1+\lambda_2)}{(1-\lambda_1)(1-\lambda_2)}.
$$
\end{cor}
\begin{proof} By using Lemma \ref{qlem2.9}, we get that $I_X+T^H\delta T$ is invertible in $H(X,X)$ and
\begin{eqnarray}\label{qeq3.62j}
\|(I_X+ T^H \delta T)^{-1}\|\leq \dfrac{1+\lambda_2}{1-\lambda_1},\qquad \|I_X+ T^H \delta T\|\leq
\dfrac{1+\lambda_1}{1-\lambda_2}.
\end{eqnarray}

From Theorem \ref{fmthm1.21}, we see $\bar{T}^H=T^H(I_Y+\delta T T^H)^{-1}=(I_X+ T^H\delta T)^{-1}T^H$ is well--defined.
Now we can compute
\begin{align}\label{qeq3.6j}
\dfrac{\|\bar{T}^H-T^H\|}{\|T^H\|}&\leq \dfrac{\|(I_X+ T^H\delta T)^{-1}T^H-T^H\|}{\|T^H\|}\nonumber\\
&\leq\dfrac{\|(I_X+ T^H\delta T)^{-1}[I_X-(I_X+ T^H\delta T)]T^H\|}{\|T^H\|}\nonumber\\
&\leq \|(I_X+ T^H\delta T)^{-1}\|\|T^H\delta T\|.
\end{align}

Since $\lambda_2\in [0, 1)$, then from the second inequality in \eqref{qeq3.62j}, we get that $\|T^H \delta T\|\leq \dfrac{2+\lambda_1}{1-\lambda_2}$. Now, by using \eqref{qeq3.6j} and \eqref{qeq3.62j}, we can obtain
$$\dfrac{\|\bar{T}^H-T^H\|}{\|T^H\|}\leq \dfrac{(2+\lambda_1)(1+\lambda_2)}{(1-\lambda_1)(1-\lambda_2)}.$$
This completes the proof.
\end{proof}

\begin{cor}
Let $T\in B(X,Y)$ with $\R(T)$ closed. Assume that $\R(T)$ and $\N(T)$ are Chebyshev subspaces in $Y$ and $X$,
respectively. Let $\delta T\in B(X,Y)$ such that $\R(\delta T)\subset\R(T)$, $\N(T)\subset\N(\delta T)$ and
$\|T^M\delta T\|<1$. Put $\bar{T}=T+\delta T$. If $T^M$ is quasi--additive on $\R(T)$, then $\bar T^M=T^M(I_Y+\delta TT^M)^{-1}=
(I_X+T^M\delta T)^{-1}T^M$ with
$$\dfrac{\|\bar T^M-T^M\|}{\|T^M\|}\le\dfrac{\|T^M\delta T\|}{1-\|T^M\delta T\|}.$$
\end{cor}

\begin{proof}From $\R(\delta T)\subset\R(T)$ and $\N(T)\subset\N(\delta T)$, we get that $\pi_{\R(T)}\delta T=\delta T$
and $\delta T\pi_{\N(T)}=0$, that is, $TT^M\delta T=\delta T=\delta TT^MT$. Consequently,
\begin{equation}\label{xxx}
\bar T=T+\delta T=T(I_X+T^M\delta T)=(I_Y+\delta TT^M)T
\end{equation}
Since $T^M$ is quasi--additive on $\R(T)$ and $\|T^M\delta T\|<1$, we get that $I_X+T^M\delta T$ and $I_Y+\delta TT^M$
are all invertible in $H(X,X)$. So from (\ref{xxx}), we have $\R(\bar T)=\R(T)$ and $\N(\bar T)=\N(T)$ and hence
$\bar T^H=T^M(I_Y+\delta TT^M)^{-1}=(I_X+T^M\delta T)^{-1}T^M$ by Theorem \ref{fmthm1.21}. Finally, by Corollary \ref{cor2b},
\begin{align*}
\bar T^M&=(I_X-\pi_{\N(\bar T)})\bar T^H\pi_{\R(\bar T)}=(I_X-\pi_{\N(T)})T^M(I_Y+\delta TT^M)^{-1}\pi_{\R(T)}\\
&=(I_X+T^M\delta T)^{-1}T^M\pi_{\R(T)}=(I_X+T^M\delta T)^{-1}T^M=T^M(I_Y+\delta TT^M)^{-1}
\end{align*}
and then
$$
\|\bar T^M-T^M\|\le\|(I_X-T^M\delta T)^{-1}-I_X\|\|T^M\|\leq\frac{\|T^M\delta T\|\|T^M\|}{1-\|T^M\delta T\|}.
$$
The proof is completed.
\end{proof}

\bibliography{JAMS-paper}

\end{document}